\documentclass[11pt,twoside]{amsart}
\usepackage{amsmath, amsthm, amscd, amsfonts, amssymb, graphicx, color}
\usepackage[bookmarksnumbered, colorlinks, plainpages]{hyperref}

\newtheorem{thm}{Theorem}[section]
\newtheorem{cor}[thm]{Corollary}
\newtheorem{lem}[thm]{Lemma}

\newtheorem{rem}[thm]{\bf{Remark}}

\numberwithin{equation}{section}

\def\pn{\par\noindent}

\begin{document}

\leftline{ \scriptsize \it International Journal of Group Theory  Vol. {\bf\rm XX} No. X {\rm(}201X{\rm)}, pp XX-XX.}

\vspace{1.3 cm}

\title{On Varietal Capability of Infinite Direct Products of Groups}
\author{Hanieh Mirebrahimi$^*$ and Behrooz Mashayekhy}

\thanks{{\scriptsize
\hskip -0.4 true cm MSC(2010): Primary: 20E10; Secondary: , 20K25, 20E34, 20D15, 20F18.
Received: 30 April 2009, Accepted: 21 June 2010.\\
$*$Corresponding author
\newline\indent{\scriptsize $\copyright$ 2011 University of Isfahan}}}

\maketitle

\begin{center}
Communicated by\;
\end{center}

\begin{abstract}  Recently, the authors gave some conditions under which a direct product
of finitely many groups is $\mathcal{V}-$capable if and only if each of its
factors is $\mathcal{V}-$capable for some varieties $\mathcal{V}$. In this paper, we extend this fact to any infinite direct product of groups. Moreover, we conclude some results for $\mathcal{V}-$capability of direct products of infinitely many groups in varieties of abelian, nilpotent and polynilpotent groups.
\end{abstract}

\vskip 0.2 true cm


\pagestyle{myheadings}
\markboth{\rightline {\scriptsize  Mirebrahimi and Mashayekhy}}
         {\leftline{\scriptsize On Varietal Capability of Infinite Direct Products}}

\bigskip
\bigskip


\section{\bf Introduction}
\vskip 0.4 true cm

R. Baer \cite{B} initiated an investigation of the question "which
conditions a group $G$ must fulfill in order to be the group of
inner automorphisms of a group $E$?", that is $G\cong E/Z(E)$.
Following M. Hall and J. K. Senior \cite{HS}, such a group $G$ is
called \textit{capable}. Baer \cite{B} determined all capable groups which are
direct sums of cyclic groups. As P. Hall \cite{H} mentioned,
characterizations of capable groups are important in classifying
groups of prime-power order.

F. R. Beyl, U. Felgner and P. Schmid \cite{BFS} proved that every group
$G$ possesses a uniquely determined central subgroup $Z^*(G)$
which is minimal subject to being the image in $G$ of the center
of some central extension of $G$. This $Z^*(G)$ is characteristic
in $G$ and is the image of the center of every stem cover of $G$.
Moreover, $Z^*(G)$ is the smallest central subgroup of $G$  whose
factor group is capable \cite{BFS}. 
Hence $G$ is capable if and only if
$Z^*(G)=1$. 
They showed that the class of all capable groups is
closed under the direct products. 
Also, they presented a condition
in which the capability of a direct product of finitely many of
groups implies the capability of each of the factors. 
Moreover,
they proved that if $N$ is a central subgroup of $G$, then
$N\subseteq Z^*(G)$ if and only if the mapping $M(G)\rightarrow
M(G/N)$ induced by the natural epimorphism, is monomorphism. 

Then M. R. R. Moghadam and S. Kayvanfar \cite{MK} generalized the
concept of capability to $\mathcal{V}-$capability for a group
$G$. They introduced the subgroup $(V^*)^*(G)$ which is
associated with the variety $\mathcal{V}$ defined by a set of
laws $V$ and a group $G$ in order to establish a necessary and
sufficient condition under which $G$ can be
$\mathcal{V}-$capable. 
They also showed that the class of all
$\mathcal{V}-$capable groups is closed under the direct products. 
Moreover, they exhibited a close relationship between the groups
$\mathcal{V}M(G)$ and $\mathcal{V}M(G/N)$, where $N$ is a normal
subgroup contained in the marginal subgroup of $G$ with respect
to the variety $\mathcal{V}$. Using this relationship, they gave
a necessary and sufficient condition for a group $G$ to be
$\mathcal{V}-$capable. 

The authors \cite{MM} presented some conditions in which
the $\mathcal{V}-$capablity of a direct product of finitely many
groups implies the $\mathcal{V}-$capablity of each of its factors.
In this paper, we extend this fact to direct product of an infinite family of groups. Also, we deduce some new results about the  $\mathcal{V}-$capability of direct product of infinitely many groups, where $\mathcal{V}$ is the variety of abelian, nilpotent, or polynilpotent groups.


\section{\bf {\bf \em{\bf Main Results}}}
\vskip 0.4 true cm

Suppose that $\mathcal{V}$ is a variety of groups defined by the set of
laws $V$. A group $G$ is said to be $\mathcal{V}-$capable if
there exists a group $E$ such that $G\cong E/V^*(E)$, where $V^*(E)$ is the marginal subgroup of $E$, which is defined as follows \cite{K}: $$\{g\in E\ |\ v(x_1, x_2, ..., x_n)=v(x_1, x_2, ..., gx_i, x_{i+1}, ..., x_n)\ $$$$\forall x_1, x_2, ..., x_n\in E,\ \forall i\in \{1, 2, ..., n\}\}.$$ If
$\psi:E\rightarrow G$ is a surjective homomorphism with $ker
\psi\subseteq V^*(E)$, then the intersection of all subgroups of
the form $\psi(V^*(E))$ is denoted by $(V^*)^*(G)$. It is obvious
that $(V^*)^*(G)$ is a characteristic subgroup of $G$ contained
in $V^*(G)$. If $\mathcal{V}$ is the variety of abelian groups,
then the subgroup $(V^*)^*(G)$ is the same as $Z^*(G)$ and in
this case $\mathcal{V}-$capability is equal to capability \cite{MK}. In the following, there are some results which we need them in sequel.

\begin{thm}\label{2.1} \cite{MK} (i) A group $G$ is $\mathcal{V}-$capable if and only if
$(V^*)^*(G)=1$.\\
(ii) If $\{G_i\}_{i\in I}$ is a family of groups, then $(V^*)^*(\prod_{i\in I}^{}G_i)\subseteq \prod_{i\in
I}(V^*)^*(G_i).$
\end{thm}

As a consequence, if the $G_i$'s are $\mathcal{V}-$capable groups,
then $G=\prod_{i\in I}^{}G_i$ is also $\mathcal{V}-$capable. In the above theorem, the equality does not hold in
general (see Example $2.3$ of \cite{MM}).

\begin{thm}\label{2.2} \cite{MK} Let $\mathcal{V}$ be a variety of groups with a set of laws $V$. Let $G$ be a group and $N$ be a normal subgroup with the property $N\subseteq V^*(G)$. Then $N\subseteq (V^*)^*(G)$ if and only if the
homomorphism induced by the natural map
${\mathcal{V}}M(G)\rightarrow {\mathcal{V}}M(G/N)$ is a
monomorphism.
\end{thm}

We recall that the Baer-invariant of a group $G$, with the free presentation $F/R$, with respect to the
variety ${\mathcal  V}$, denoted by ${\mathcal V}M(G)$, is
$${\mathcal V}M(G)=\frac{R\cap V(F)}{[R V^*F]}\ ,$$
where $V(F)$ is the verbal subgroup of $F$ with respect to $\mathcal V$ and
$$[RV^*F]=<v(f_1,\ldots ,f_{i-1},f_ir,f_{i+1},\ldots,f_n)v(f_1,\ldots,f_i,
\ldots,f_n)^{-1} | r\in R, $$ $$f_i\in F,v\in V,1\leq i\leq n,n\in {\bf N}>\ .$$
It is known that the Baer-invariant of a group $G$ is always abelian and
independent of the choice of the presentation of $G$.  Also if $\mathcal V$ is the variety of abelian groups, then the
Baer-invariant of $G$ will be
${R\cap F'}/{[R,F]}\cong M(G)$,
where $M(G)$ is the Schur multiplier of $G$ (see \cite{K}).

\begin{thm}\label{2.3} \cite{MM} Let $\mathcal{V}$ be a variety, $A$ and $B$ be two
groups with ${\mathcal{V}}M({A\times B})\cong
{\mathcal{V}}M({A})\times{\mathcal{V}}M({B})$, then
$(V^*)^*(A\times B)=(V^*)^*(A)\times (V^*)^*(B)$. Consequently,
$A\times B$ is $\mathcal{V}$-capable if and only if $A$ and $B$
are both $\mathcal{V}$-capable.
\end{thm}

\begin{thm}\label{2.4} \cite{M1}Let $\{G_i; \phi_i^j,I\}$ be a directed system of groups. Then, for a given variety ${\mathcal{V}}$, the Baer-invariant preserves direct limit, that is ${\mathcal{V}}M(\displaystyle{\lim_{\rightarrow} G_i})=\displaystyle{\lim_{\rightarrow}{\mathcal{V}}M(G_i)}$.
\end{thm}

\begin{lem}\label{3.1}
For any family of groups $\{G_i\}_{i\in I}$, consider the directed system $\{\mathcal{G}_{I_{\lambda}}, \phi^{\lambda'}_{\lambda}, \Lambda\}$ consisting of all finite direct products $\mathcal{G}_{I_{\lambda}}=\prod_{i\in I_{\lambda}} G_i$ ($I_{\lambda}$ is a finite subset of $I$), with the natural embedding homomorphisms $\phi^{\lambda'}_{\lambda}:\mathcal{G}_{I_{\lambda}}\rightarrow \mathcal{G}_{I_{\lambda'}}$ ($I_{\lambda}\subseteq I_{\lambda'}$). Also, the index set $\Lambda$ is ordered in a directed way so that for any $\lambda, \lambda'\in \Lambda$, $\lambda\leq\lambda'$ if and only if $I_{\lambda}\subseteq I_{\lambda'}$. Then the direct product $\mathcal{G}_{I}=\prod_{i\in I} G_i$ is a direct limit of this directed system.
\end{lem}
\begin{proof}
Let $\mathcal{G}=\displaystyle{\lim_{\rightarrow}\mathcal{G}_{I_{\lambda}}}$ be a direct limit of this directed system, with homomorphisms $\phi_{\lambda}:\mathcal{G}_{I_{\lambda}}\rightarrow \mathcal{G}$. Also, for any $\lambda\in \Lambda$, consider the embedding homomorphism $\tau_{\lambda}:\mathcal{G}_{I_{\lambda}}\rightarrow \mathcal{G}_I$. Clearly, for any  $\lambda, \lambda'\in \Lambda$ with $\lambda\leq\lambda'$, $\tau_{\lambda'}\tau^{\lambda'}_{\lambda}=\tau_{\lambda}$. Now, by universal property of $\mathcal{G}$, there exists a unique homomorphism $\phi:\mathcal{G}\rightarrow\mathcal{G}_I$ such that for any $\lambda\in\Lambda$, $\phi\phi_{\lambda}=\tau_{\lambda}$. To define the inverse homomorphism $\tau:\mathcal{G}_I\rightarrow\mathcal{G}$, recall that for any $x=\{x_i\}_{i\in I}\in \mathcal{G}_I$, there exists a finite subset $I_{\lambda}$ of $I$ that for any $i\in I\backslash I_{\lambda}$, $x_i$ is trivial in $G_i$. Hence we can consider $x$ as an element of $\mathcal{G}_{I_{\lambda}}$ and define $\tau(x)=\phi_{\lambda}(x)$. It is easy to see that for any  $\lambda\in\Lambda$, $\tau\tau_{\lambda}=\phi_{\lambda}$. Finally, we see that for any $x\in \mathcal{G}_I$, $\phi\tau(x)=\phi\phi_{\lambda}(x)$, for some ${\lambda}\in \Lambda$; and so $\phi\tau(x)=\tau_{\lambda}(x)=x$. Conversely, the equation $\tau\phi=id_{\mathcal{G}}$ holds because of the universal property of the direct limit $\mathcal{G}$.
\end{proof}

By the above notations, we conclude that $\prod_{i\in I} G_i$, $\prod_{i\in I} V^{**}(G_i)$, and $\prod_{i\in I} G_i/V^{**}(G_i)$ are direct limits of directed systems $\{\prod_{i\in I_{\lambda}} G_i, \phi^{\lambda'}_{\lambda}, \Lambda\}$, $\{\prod_{i\in I_{\lambda}} V^{**}(G_i), \bar{\phi}^{\lambda'}_{\lambda}, \Lambda\}$, and $\{\prod_{i\in I_{\lambda}} G_i/V^{**}(G_i), \psi^{\lambda'}_{\lambda}, \Lambda\}$ respectively, where $\bar{\phi}^{\lambda'}_{\lambda}$'s are restrictions of $\phi^{\lambda'}_{\lambda}$'s and $\psi^{\lambda'}_{\lambda}$'s are quotient homomorphisms induced by $\phi^{\lambda'}_{\lambda}$'s.\\

Now, suppose that $\{G_i\}_{i\in I}$ is a family of groups in which for any $G_i$ and $G_j$ $(i,j\in I)$, ${\mathcal{V}}M({G_i\times G_j})\cong{\mathcal{V}}M({G_i})\times{\mathcal{V}}M({G_j})$. By Theorem \ref{2.3},  $\prod_{i\in I_{\lambda}}(V^*)^*(G_i)\subseteq(V^*)^*(\prod_{i\in I_{\lambda}} G_i)$, for any finite subset $I_{\lambda}$ of $I$. Thus, using Theorem \ref{2.2}, we have the following monomorphism $${\mathcal{V}}M(\prod_{i\in I_{\lambda}} G_i)\hookrightarrow {\mathcal{V}}M(\frac{\prod_{i\in I_{\lambda}} G_i}{\prod_{i\in I_{\lambda}}(V^*)^*(G_i)}).$$ By the fact that direct limit of a directed system preserves exactness of a sequence \cite{M1}, we obtain the following monomorphism
$$\displaystyle{\lim_{\rightarrow}{\mathcal{V}}M(\prod_{i\in I_{\lambda}} G_i)}\hookrightarrow \displaystyle{\lim_{\rightarrow}}{\mathcal{V}}M(\frac{\prod_{i\in I_{\lambda}} G_i}{\prod_{i\in I_{\lambda}}(V^*)^*(G_i)}).$$ Using Theorem \ref{2.4}, we conclude the monomorphism $${\mathcal{V}}M(\displaystyle{\lim_{\rightarrow}\prod_{i\in I_{\lambda}} G_i})\hookrightarrow {\mathcal{V}}M(\displaystyle{\lim_{\rightarrow}\frac{\prod_{i\in I_{\lambda}} G_i}{\prod_{i\in I_{\lambda}}(V^*)^*(G_i)})},$$ and so we have the monomorphism $${\mathcal{V}}M(\prod_{i\in I} G_i)\hookrightarrow {\mathcal{V}}M(\frac{\prod_{i\in I} G_i}{\prod_{i\in I}(V^*)^*(G_i)}).$$ Finally, by Theorem \ref{2.2}, we conclude that $${\prod_{i\in I}(V^*)^*(G_i)}\subseteq {(V^*)^*(\prod_{i\in I}G_i)}.$$
Using these notes, we deduce the following theorem.

\begin{thm}\label{thm 3.2} Let $\mathcal{V}$ be a variety, $\{G_i\}_{i\in I}$ be a family of
groups such that for any $i,j\in I$, ${\mathcal{V}}M({G_i\times G_j})\cong
{\mathcal{V}}M({G_i})\times{\mathcal{V}}M({G_j})$. Then $(V^*)^*(\prod_{i\in I}G_i)=\prod_{i\in I}(V^*)^*(G_i)$.  Consequently, $\prod_{i\in I}G_i$ is $\mathcal{V}$-capable if and only if each $G_i$ is $\mathcal{V}$-capable.
\end{thm}

\begin{rem}
(i) In the above theorem, the sufficient condition $${\mathcal{V}}M({A\times B})\cong
{\mathcal{V}}M({A})\times{\mathcal{V}}M({B})$$ is not necessary (see Example $2.3 (iii)$ of \cite{MM}). Also, this condition is essential and can not be omitted (see Example $2.3 (i),(ii)$  of \cite{MM}).\\
(ii) It is known that for varieties of abelian and nilpotent groups, and for any groups $A$ and $B$, ${\mathcal{V}}M(A\times
B)\cong {\mathcal{V}}M(A)\times {\mathcal{V}}M(B)\times T$, where
$T$ is an abelian group whose elements are tensor products of the elements of $A^{ab}$ and $B^{ab}$ \cite{E}, \cite{M2}. Hence in these known varieties, the isomorphism ${\mathcal{V}}M({A\times B})\cong
{\mathcal{V}}M({A})\times{\mathcal{V}}M({B})$ holds, where both $A^{ab}$ and $B^{ab}$ have finite exponent with  $(exp(A^{ab}),exp(B^{ab}))=1$.
\end{rem}

In the following, using the main theorem and the above remark, we deduce some corollaries which are  generalizations of some results of \cite{MM} (Remark $2.4 (ii)$, Corollary $2.5$ and Example $2.2$).

\begin{cor} Let $\{G_i\}_{i\in I}$ be a family of
groups whose abelianizations have mutually coprime exponents. Then $\prod_{i\in I} G_i$ is capable ($\mathcal{N}_{c}$-capable) if and only if each $G_i$ is capable ($\mathcal{N}_{c}$-capable).
\end{cor}

\begin{cor} Suppose that $\{G_i\}_{i\in I}$ is a family of
groups whose abelianizations have mutually coprime exponents.
If $\prod_{i\in I} G_i$ is nilpotent of class at most $c_1$, then it
is $\mathcal{N}_{c_1,\cdots,c_s}$-capable if and only if every
$G_i$ is $\mathcal{N}_{c_1,\cdots,c_s}$-capable.
\end{cor}

\begin{cor} If $\{G_i\}_{i\in I}$ is a family of perfect groups, then
$\prod_{i\in I} G_i$ is $\mathcal{V}$-capable if and only if each
$G_i$ is $\mathcal{V}$-capable, where $\mathcal{V}$ may be each
of these three varieties, 1.variety of abelian groups, 2.variety of
nilpotent groups, 3.variety of polynilpotent groups.
\end{cor}




\bigskip
\bigskip


{\footnotesize \pn{\bf Hanieh Mirebrahimi}\; \\
{Department of Pure Mathematics, Center of Excellence in Analysis on Algebraic Structures,
Ferdowsi University of Mashhad,
P. O. Box 1159-91775, Mashhad, Iran.}\\
          {\tt Email: h${\_}$mirebrahimi@um.ac.ir}

{\footnotesize \pn{\bf Behrooz Mashayekhy}\; \\
{Department of Pure Mathematics, Center of Excellence in Analysis on Algebraic Structures,
Ferdowsi University of Mashhad,
P. O. Box 1159-91775, Mashhad, Iran.}\\
          {\tt Email: bmashf@um.ac.ir}

\end{document}